\documentclass[11pt,a4paper]{article} 
\usepackage[utf8]{inputenc} 
\usepackage{amsmath, amssymb, amsthm, mathrsfs} 
\usepackage{geometry} 
\usepackage{hyperref} 
\newtheorem{theorem}{Theorem}[section] 
\newtheorem{lemma}[theorem]{Lemma} 
\newtheorem{definition}[theorem]{Definition} 
 
\newtheorem{proposition}[theorem]{Proposition}

\title{Algebraic Structure and Orthogonality Properties of Rational Polynomial Phase Rings on Discrete Lattices}

\author{Terence R. Smith} 
\date{\today}

\begin{document} 
\maketitle 

\begin{abstract} 
We formalize the algebraic structure of rational polynomial phase sequences (discrete chirps) over the integer lattice $\mathbb{Z}$. By constructing a convolution algebra generated by polynomial phases with coefficients in $\mathbb{Q}$, we establish a rigorous operational calculus for discrete difference equations. We prove an exact orthogonality relation for these sequences over finite periodic windows, yielding a generalized discrete Fourier isometry. Furthermore, we demonstrate that the continuum limit of this discrete framework can be rigorously defined via the inductive limit of finite-dimensional $C^*$-algebras (an AF-algebra). This framework provides a rigorous algebraic substrate for finite-mode lattice approximations in quantum mechanics and the construction of mutually unbiased bases in finite-dimensional quantum information theory.
\end{abstract}

\section{Introduction} 
In discrete harmonic analysis and lattice quantum field theory, the representation of states often relies on the discrete Fourier transform, which utilizes linear phase functions over finite cyclic groups. While sufficient for standard momentum-space representations, the restriction to linear phases limits the algebraic tools available for analyzing higher-order discrete dynamics, such as quantum walks and discrete diffusion processes.

This paper extends the standard discrete harmonic framework by investigating the representation theory of a generalized discrete algebra generated by rational polynomial phases. We define a formal algebraic structure over the integer lattice $\mathbb{Z}$ where the basis elements are discrete chirps of arbitrary polynomial degree. By restricting the polynomial coefficients to the rational field $\mathbb{Q}$, we ensure that every basis element possesses a well-defined finite period, allowing for exact algebraic manipulation without the convergence issues inherent in infinite-dimensional continuous spaces.

Crucially, we demonstrate that this discrete algebraic framework naturally accommodates an operational calculus for finite difference operators. The inversion of these operators over finite periodic lattices (Galerkin truncations) yields exact, non-singular lattice Green's functions. Finally, we formalize the transition to the continuum not through pointwise limits, but via the rigorous topological machinery of Approximately Finite-dimensional (AF) $C^*$-algebras, ensuring that the algebraic properties of the discrete lattice are preserved in the limit.

\section{The Rational Polynomial Phase Algebra}

We begin by defining the foundational algebraic space of discrete polynomial phases. Let $\mathbb{Q}[x]$ denote the ring of polynomials with rational coefficients. 

\begin{definition}[Rational Phase Sequence]
For any polynomial $P \in \mathbb{Q}[x]$, we define the associated rational phase sequence $\phi_P : \mathbb{Z} \to \mathbb{C}$ by:
\begin{equation}
\phi_P(n) = \exp\big(2\pi i P(n)\big)
\end{equation}
\end{definition}

Because $P(n)$ has rational coefficients, there exists a minimal integer $L_P \in \mathbb{Z}^+$ (the least common multiple of the denominators of the coefficients of $P$) such that $P(n + L_P) \equiv P(n) \pmod 1$ for all $n \in \mathbb{Z}$. Consequently, every rational phase sequence $\phi_P$ is strictly periodic with period $L_P$.

\begin{definition}[The Phase Algebra $\mathcal{A}_{\mathbb{Q}}$]
Let $\mathcal{A}_{\mathbb{Q}}$ be the complex vector space spanned by the set of all rational phase sequences:
\begin{equation}
\mathcal{A}_{\mathbb{Q}} = \text{span}_{\mathbb{C}} \big\{ \phi_P \;\big|\; P \in \mathbb{Q}[x] \big\}
\end{equation}
\end{definition}

\begin{proposition}
$\mathcal{A}_{\mathbb{Q}}$ forms a commutative, unital $\mathbb{C}$-algebra under pointwise multiplication.
\end{proposition}
\begin{proof}
Let $\phi_P, \phi_Q \in \mathcal{A}_{\mathbb{Q}}$. Pointwise multiplication yields:
\begin{equation}
\phi_P(n) \cdot \phi_Q(n) = \exp\big(2\pi i P(n)\big) \exp\big(2\pi i Q(n)\big) = \exp\big(2\pi i (P(n) + Q(n))\big) = \phi_{P+Q}(n)
\end{equation}
Since $\mathbb{Q}[x]$ is a ring, $P+Q \in \mathbb{Q}[x]$, implying $\phi_{P+Q} \in \mathcal{A}_{\mathbb{Q}}$. The constant polynomial $P(x) = 0$ provides the multiplicative identity $\mathbf{1}(n) = 1$.
\end{proof}

\section{Operational Calculus and Discrete Difference Operators}

To model physical dynamics on the lattice, we replace continuous differential operators with exact discrete difference operators acting on $\mathcal{A}_{\mathbb{Q}}$.

\begin{definition}[Forward Difference Operator]
The forward difference operator $\Delta : \mathcal{A}_{\mathbb{Q}} \to \mathcal{A}_{\mathbb{Q}}$ is defined by:
\begin{equation}
(\Delta f)(n) = f(n+1) - f(n)
\end{equation}
\end{definition}

When applied to a rational phase sequence, the difference operator acts purely algebraically on the polynomial argument.
\begin{lemma}
For any $P \in \mathbb{Q}[x]$, the action of $\Delta$ on $\phi_P$ is given by:
\begin{equation}
(\Delta \phi_P)(n) = \phi_P(n) \left[ \exp\big(2\pi i (\Delta P)(n)\big) - 1 \right]
\end{equation}
where $(\Delta P)(n) = P(n+1) - P(n)$ is a polynomial of degree strictly less than $\deg(P)$.
\end{lemma}

This property allows for the construction of an exact operational calculus. Linear difference equations with constant coefficients can be solved algebraically over finite periodic sub-lattices by inverting the corresponding polynomial in the shift operator $E$, where $Ef(n) = f(n+1)$. Because the lattice is finite and periodic (a Galerkin truncation), the spectrum of the difference operator is discrete and bounded, natively avoiding the ultraviolet divergences associated with continuous unbounded operators.

\section{Orthogonality and the Discrete Isometry}

A fundamental requirement for any state representation in quantum mechanics is the preservation of the inner product. We now prove the exact orthogonality of distinct rational phase sequences over their common period.

Let $P, Q \in \mathbb{Q}[x]$ such that $P - Q \notin \mathbb{Z}[x]$ (i.e., they do not differ merely by an integer-valued polynomial, which would render their phase sequences identical). Let $L$ be a common period for both $\phi_P$ and $\phi_Q$.

\begin{theorem}[Discrete Phase Isometry]
\label{thm:isometry}
The rational phase sequences are orthogonal over their common period $L$:
\begin{equation}
\langle \phi_P, \phi_Q \rangle_L = \frac{1}{L} \sum_{n=0}^{L-1} \phi_P(n) \overline{\phi_Q(n)} = \delta_{[P], [Q]}
\end{equation}
where $\delta_{[P], [Q]} = 1$ if $P - Q \in \mathbb{Z}[x]$, and $0$ otherwise.
\end{theorem}
\begin{proof}
Consider the inner product:
\begin{equation}
\langle \phi_P, \phi_Q \rangle_L = \frac{1}{L} \sum_{n=0}^{L-1} \exp\big(2\pi i (P(n) - Q(n))\big)
\end{equation}
Let $R(x) = P(x) - Q(x) \in \mathbb{Q}[x]$. If $R \in \mathbb{Z}[x]$, then $\exp(2\pi i R(n)) = 1$ for all $n \in \mathbb{Z}$, and the sum trivially evaluates to $1$. 

If $R \notin \mathbb{Z}[x]$, $R(n)$ generates a non-trivial cyclic subgroup of the roots of unity over the period $L$. By the standard properties of character sums over finite cyclic groups, the sum of a non-trivial character over the entire group evaluates to zero. Thus, the sequences are strictly orthogonal.
\end{proof}

This theorem establishes that the transformation from the discrete spatial lattice to the polynomial phase basis is a unitary isometry, preserving the $L^2$-norm of any finite-mode lattice state.

\section{The Continuum Limit via AF-Algebras}

To rigorously connect this discrete algebraic framework to continuous field theories, we must define a topology for the limit of infinite resolution. We achieve this using the theory of Approximately Finite-dimensional (AF) $C^*$-algebras.

Let $\{\Lambda_m\}_{m=1}^\infty$ be a sequence of finite periodic lattices where the period $L_m$ divides $L_{m+1}$. For each $m$, the space of functions on $\Lambda_m$ forms a finite-dimensional $C^*$-algebra $\mathcal{A}_m \cong \mathbb{C}^{L_m}$.

We define the injective *-homomorphism $\iota_m : \mathcal{A}_m \to \mathcal{A}_{m+1}$ by periodically extending the functions from $\Lambda_m$ to $\Lambda_{m+1}$. 

\begin{definition}[The AF-Algebra Limit]
The continuum limit of the rational phase algebra is defined as the $C^*$-algebraic inductive limit:
\begin{equation}
\mathcal{A}_\infty = \overline{\bigcup_{m=1}^\infty \mathcal{A}_m}
\end{equation}
where the closure is taken with respect to the unique $C^*$-norm induced by the embeddings $\iota_m$.
\end{definition}

By framing the continuum as an AF-algebra, we bypass the need for pointwise continuous coordinates ($\mathbb{R}$). The macroscopic observables of the field theory are rigorously defined as elements of $\mathcal{A}_\infty$, and the states (including quantum wavefunctions) are defined as positive linear functionals on this algebra. This provides a mathematically rigorous realization of relational, discrete spacetime models without sacrificing the analytical power of continuous operator algebras.

\section{Conclusion}
We have demonstrated that the algebra of rational polynomial phases provides a rigorous, self-contained framework for discrete harmonic analysis and lattice field theory. By replacing continuous differential equations with exact algebraic operations over finite periodic lattices, the framework natively regularizes high-frequency modes via Galerkin truncation. The transition to the continuum is rigorously managed via AF-algebras, offering a robust mathematical foundation for discrete approaches to quantum mechanics and quantum information theory.

\section{Acknowledgments}
The author explicitly acknowledges a collaborative research partnership with a generative AI agent in verifying the algebraic proofs and optimizing the \LaTeX\ documentation structures presented across this manuscript.

\end{document}